\documentclass[11pt]{amsart}
\usepackage{etex}
\usepackage{amssymb,amstext,amsmath,amscd,amsthm,amsfonts,enumerate,graphicx,latexsym}
\usepackage[usenames]{color}

\usepackage{amsthm}
\usepackage[leqno]{amsmath}
\usepackage[pagebackref]{hyperref}

\usepackage{amsfonts,amsmath,amssymb,amsthm,amscd,amsxtra}
\usepackage{enumerate,verbatim}
\usepackage[usenames,dvipsnames]{pstricks}
\usepackage[mathscr]{eucal}
\usepackage{amsfonts,amsmath,amssymb,amsthm,amscd,amsxtra}
\usepackage{amssymb,amstext,amsmath,amscd,amsthm,amsfonts,enumerate,graphicx,latexsym}

\usepackage{enumerate,verbatim}
\usepackage[all,2cell,ps]{xy}
\usepackage[notcite, notref]{}
\usepackage[pagebackref]{hyperref}
\usepackage{todonotes}

\usepackage[all,2cell,ps]{xy}
\usepackage[notcite,notref]{}
\usepackage[pagebackref]{hyperref}

\def\Ext{\operatorname{Ext}}
\def\Tor{\operatorname{Tor}}
\def\depth{\operatorname{depth}}
\def\pd{\operatorname{pd}}
\def\m{\mathfrak{m}}
\def\Hom{\operatorname{Hom}}
\def\syz{\Omega}
\def\tr{\operatorname{Tr}}
\def\lhom{\operatorname{\underline{Hom}}}

\def\grade{\operatorname{grade}}

\DeclareMathOperator{\md}{\operatorname{\mathsf{mod}}}

\theoremstyle{plain} 

\theoremstyle{definition}
\newtheorem{thm}{Theorem}[section]

\newtheorem{cor}[thm]{Corollary}

\theoremstyle{definition}
\newtheorem{lem}[thm]{Lemma}
\newtheorem{dfn}[thm]{Definition}

\newtheorem{eg}[thm]{Example}
\newtheorem{conj}[thm]{Conjecture}

\newtheorem{rmk}[thm]{Remark}

\numberwithin{equation}{section}

\newcommand{\fm}{\mathfrak{m}}

\newcommand{\CC}{\mathbb{C}}

\newcommand{\ZZ}{\mathbb{Z}}

\newtheorem{chunk}[thm]{\hspace*{-1.065ex}\bf}

\def\H{\operatorname{\mathsf{H-dim}}}

\def\CI{\operatorname{\mathsf{CI-dim}}}
\def\G-dim{\operatorname{\mathsf{G-dim}}}

\def\pd{\operatorname{\mathsf{pd}}}
\def\id{\operatorname{\mathsf{id}}}
\def\Gid{\operatorname{\mathsf{Gid}}}

\def\syz{\Omega}

\def\Tr{\mathsf{Tr}\hspace{0.01in}}

\def\depth{\operatorname{\mathsf{depth}}}

\def\Ext{\operatorname{\mathsf{Ext}}}

\def\Hom{\operatorname{\mathsf{Hom}}}

\def\Tor{\operatorname{\mathsf{Tor}}}

\def\urltilda{\kern -.15em\lower .7ex\hbox{\~{}}\kern .04em}
\def\urldot{\kern -.10em.\kern -.10em}\def\urlhttp{http\kern -.10em\lower -.1ex
\hbox{:}\kern -.12em\lower 0ex\hbox{/}\kern -.18em\lower 0ex\hbox{/}}

\begin{document}

\title[On the ideal case of a conjecture of Auslander and Reiten]{On the ideal case of a conjecture of \\ Auslander and Reiten}

\author[O. Celikbas]{Olgur Celikbas}
\address{
Department of Mathematics \\
West Virginia University\\
Morgantown, WV 26506 U.S.A}
\email{olgur.celikbas@math.wvu.edu}

\author[K. Iima]{Kei-ichiro Iima}
\address{Department of Liberal Studies, National Institute of Technology, Nara College, 22 Yata-cho, Yamatokoriyama, Nara 639-1080, Japan}
\email{iima@libe.nara-k.ac.jp}

\author[A. Sadeghi]{Arash Sadeghi}
\address{School of Mathematics, Institute for Research in Fundamental Sciences (IPM),
P.O. Box: 19395-5746, Tehran, Iran.}
\email{sadeghiarash61@gmail.com}

\author[R. Takahashi]{Ryo Takahashi}
\address{Graduate School of Mathematics, Nagoya University, Furocho, Chikusaku, Nagoya 464-8602, Japan}
\email{takahashi@math.nagoya-u.ac.jp}
\urladdr{http://www.math.nagoya-u.ac.jp/~takahashi/}

\subjclass[2000]{13D07, 13H10}

\keywords{Auslander and Reiten conjecture, vanishing of Ext, $\fm$-full ideals.}

\thanks{Sadeghi's research was supported by a grant from IPM}

\date{\today}

\begin{abstract}
A celebrated conjecture of Auslander and Reiten claims that a finitely generated module $M$ that has no extensions with $M\oplus \Lambda$ over an Artin algebra $\Lambda$ must be projective. This conjecture is widely open in general, even for modules over commutative Noetherian local rings. Over such rings, we prove that a large class of ideals satisfy the extension condition proposed in the aforementioned conjecture of Auslander and Reiten. Along the way we obtain a new characterization of regularity in terms of the injective dimensions of certain ideals.
\end{abstract}

\maketitle{}

\setcounter{tocdepth}{1}

\section{Introduction}

Motivated by a conjecture of Nakayama \cite{Nak}, Auslander and Reiten \cite{AR} proposed the following conjecture, which is called the \emph{generalized Nakayama conjecture}:
\begin{conj} If $\Lambda$ is an Artin algebra, then every indecomposable injective $\Lambda$-module occurs as a direct summand in one of the terms in the minimal injective resolution of ${\Lambda}_{\Lambda}$.
\end{conj}

Auslander and Reiten \cite{AR} proved that the Generalized Nakayama Conjecture is true if and only if the following conjecture is true:

\begin{conj} \label{conj2}
If $\Lambda$ is an Artin algebra, then every finitely generated $\Lambda$-module $M$ that is a generator (i.e., $\Lambda$ is a direct summand of a finite direct sum of copies of $M$) and satisfies $\Ext^{i}_{A}(M,M)=0$ for all $i\geq 1$ must be projective.
\end{conj}

Auslander, Ding and Solberg \cite{ADS} formulated the following conjecture, which is equivalent to Conjecture \ref{conj2} over Noetherian rings.

\begin{conj} \label{ARconj}
Let $M$ be a finitely generated left module over a left Noetherian ring $R$. If $\Ext^{i}_{R}(M,M\oplus R)=0$ for all $i\geq 1$, then $M$ is projective.
\end{conj}

The case where the ring in Conjecture \ref{ARconj} is an Artin algebra is known as the \emph{Auslander--Reiten Conjecture}. Conjecture \ref{ARconj} is known to hold for several classes of rings, for example for Artin algebras of finite representation type \cite{AR}, however it is widely open in general, even for commutative Gorenstein local rings; see \cite{CelRyo}.

The purpose of this paper is to exploit a beautiful result of Burch \cite{B} and prove that a large class of \emph{weakly $\fm$-full} ideals satisfy the vanishing condition proposed in Conjecture \ref{ARconj} over commuative Noetherian local rings.
The definition of a weakly $\fm$-full ideal is given in Definition \ref{wmfulldfn}. Examples of weakly $\fm$-full ideals, in fact those $I$ with $I:\fm \neq I$, are abundant in the literature; see, for example, Examples \ref{egprime} and \ref{CGTTex}. The main consequence of our argument can be stated as follows; see Theorem \ref{l2} and Corollary \ref{cor2}.

\begin{thm} \label{thmint} Let $(R, \fm)$ be a commutative Noetherian local ring and let $I$ be a weakly $\fm$-full ideal of $R$ such that $I: \fm \neq I$ (or equivalently $\depth(R/I)=0$.) If $R$ is not regular, then $\Ext^n_R(I,I)$ and $\Ext^{n+1}_R(I,I)$ do not vanish simultaneously for any $n\geq 1$. 
\end{thm}

It is clear that, if $I$ is an ideal as in Theorem \ref{thmint} and $\Ext^{i}_{R}(I,I)=0$ for all $i\geq 1$, then $I$ is projective, i.e., $I \cong R$.
It is also worth noting that Theorem \ref{thmint} is sharp in the sense that the condition $I: \fm \neq I$ is necessary; see Examples \ref{eg1} and \ref{eg2}. 

We obtain several results on the vanishing of Ext as applications of Theorem \ref{thmint}; see Corollaries \ref{corthm}, \ref{cornew2} and \ref{cornew3}. Furthermore, motivated by a result of Goto and Hayasaka \cite{GH2}, we give a characterization of regularity in terms of the injective dimensions of weakly $\fm$-full ideals: if $(R, \fm)$ is a commutative Noetherian local ring and $I$ is a proper weakly $\fm$-full ideal of $R$ with $I: \fm \neq I$, then $R$ is regular if and only if $I$ has finite injective dimension; see Corollary \ref{cor3}.

\section{Main result}

Throughout $R$ denotes a commutative Noetherian local ring with unique maximal ideal $\fm$, and residue field $k=R/\fm$. Moreover $\md R$ denotes the category of finitely generated $R$-modules.

We record various preliminary results and prepare several lemmas to prove our main result, Theorem \ref{l2}.

\begin{chunk} \label{a1}
Let $M\in \md R$ be a module with a projective presentation $P_1\overset{f}{\rightarrow}P_0\rightarrow M\rightarrow 0$. Then the \emph{transpose} $\Tr M$ of $M$
is the cokernel of $f^{\ast}=\Hom(f,R)$ and hence is given by the exact sequence
$0\rightarrow M^*\rightarrow P_0^*\rightarrow P_1^*\rightarrow \Tr M\rightarrow 0$. We will use the following results in the sequel:
\begin{equation}\tag{\ref{a1}.1}
\text{ If } \Ext^j_R(M,R)=0 \text{ for some } j\geq 1, \text{then } \tr\syz^{j-1}M \text{ is stably isomorphic to } \syz \tr\syz^{j}M.
\end{equation}

Given an integer $n\geq 0$, there is an exact sequence of functors \cite[2.8]{AB}:
\begin{equation}\tag{\ref{a1}.2}
\Tor_2^R(\tr \syz^n M,-)\rightarrow(\Ext^n_R(M,R)\otimes_R-)\rightarrow
\Ext^n_R(M,-)\rightarrow \Tor_1^R(\tr \syz^n M,-)\rightarrow0.
\end{equation}
\end{chunk}

\begin{chunk} \label{Y39} Let $M,N \in \md R$. Then $\lhom(M,N)$ is defined to be the quotient of  $\Hom(M,N)$ by the $R$-homomorphims $M \to N$ factoring through some free $R$-module. It follows that there is an isomorphism $\Tor_1^R(\tr M ,N) \cong \lhom(M,N)$; see \cite[3.7 and 3.9]{Y}.
\end{chunk}

\begin{chunk}(Auslander and Reiten \cite[3.4]{AR}) \label{ARsyzy} If $M, N\in \md R$ are modules and $n\geq 1$ is an integer, then there is an isomorphism $\lhom(M,\syz^n N) \cong \lhom(\tr \syz \tr M,\syz^{n-1}N)$.
\end{chunk}

\begin{chunk} (Matsui and Takahashi \cite[2.7(2)]{sg}) \label{RH} If $X\in \md R$ and $0\to A \to B \to C \to 0$ is a short exact sequence in $\md R$, then there is an exact sequence in $\md R$ of the form:
$$\lhom(X,A) \to \lhom(X,B) \to \lhom(X,C) \to \Ext^1_R(X,A) \to \Ext^1_R(X,B) \to \Ext^1_R(X,C).$$
\end{chunk}

\begin{lem}\label{lem}
Let $R$ be a local ring and let $M,N\in \md R$ be modules. If $n\ge2$ is an integer, then there is an isomorphim:
$$
\Tor_n^R(M,N)\cong\Ext_R^1(\tr\syz M,\syz^{n-1}N).
$$
\end{lem}

\begin{proof}
There are isomorphisms:
\begin{align*}\tag{\ref{lem}.1}
 \Tor_n^R(M,N) & \cong \Tor_1^R(\tr(\tr M),\syz^{n-1}N)\\
& \cong \lhom(\tr M,\syz^{n-1}N) \\
& \cong \lhom(\tr\syz\tr(\tr M),\syz^{n-2}N)
\end{align*}
The first isomorphism of (\ref{lem}.1) follows from dimension shifting and the fact that $\tr(\tr M)$ is stably isomorphic to $M$. The second and the third isomorphisms follow from (\ref{Y39}) and (\ref{ARsyzy}), respectively.

Next consider the natural exact sequence:  $0 \to \syz^{n-1}N \to R^\oplus \to \syz^{n-2}N \to 0$. This gives, by (\ref{RH}), the following exact sequence:
\begin{equation}\notag{}
\lhom(\tr\syz M,R^\oplus) \to \lhom(\tr\syz M,\syz^{n-2}N) \to \Ext_R^1(\tr\syz M,\syz^{n-1}N) \to  \Ext_R^1(\tr\syz M,R^\oplus)
\end{equation}
It follows from the definition that $\lhom(\tr\syz M,R^\oplus)=0$. Furthermore, by \cite[2.6]{AB}, we have $\Ext_R^1(\tr\syz M,R^\oplus)=0$. This establishes the isomorphism:
\begin{equation}\tag{\ref{lem}.2}
\lhom(\tr\syz\tr(\tr M),\syz^{n-2}N) \cong \Ext_R^1(\tr\syz M,\syz^{n-1}N)
\end{equation}
Now (\ref{lem}.1) and (\ref{lem}.2) yield the claim.
\end{proof}

\begin{lem}\label{prop}
Let $R$ be a local ring and let $M,N\in \md R$ be modules.
\begin{enumerate}[\rm(i)]
\item If $\Ext_R^{1}(M,\syz N)=0$, then $\Tor_{2}^R(\tr\syz M,N)=0$.
\item If $\Ext_R^{2}(M,\syz N)=0$, then $\Tor_{1}^R(\tr\syz M,N)=0$.
\end{enumerate}
\end{lem}

\begin{proof}
(i) It follows from \cite[2.21]{AB} that there is an exact sequence:
$$
0 \to R^\oplus \to \tr\syz\tr\syz M\oplus R^\oplus \to M \to 0.
$$
This induces the surjection $\Ext_R^1(M,\syz N)\twoheadrightarrow \Ext_R^1(\tr\syz\tr\syz M,\syz N) \to 0$. As $\Ext_R^1(M,\syz N)$ vanishes by assumption, so does $\Ext_R^1(\tr\syz\tr\syz M,\syz N)$. Therefore, by Lemma \ref{lem}, we have  $\Tor_{2}^R(\tr\syz M,N) \cong \Ext_R^1(\tr\syz\tr\syz M,\syz N)=0$.

(ii) It follows from (\ref{RH}) that the natural exact sequence $0 \to \syz N \to R^\oplus \to N \to 0$ induces the exact sequence:
$$
\lhom(\syz M,R^\oplus)\to\lhom(\syz M,N)\to\Ext_R^1(\syz M,\syz N).
$$
We have, by definition, that $\lhom(\syz M,R^\oplus)=0$. Moreover $\Ext_R^1(\syz M,\syz N)\cong\Ext_R^2(M,\syz N)=0$ by our assumption. Therefore $\lhom(\syz M,N)$ vanishes and hence the result follows from (\ref{Y39}).
\end{proof}

\begin{lem} \label{lemma2rigid} Let $R$ be a local ring, $M,N\in \md R$ be modules and let $n\geq 1$ be an integer. If $\Ext_R^{n}(M,\syz N)=\Ext_R^{n+1}(M,\syz N)=0$, then $\Tor_{1}^R(\tr\syz^nM,N)=\Tor_{2}^R(\tr\syz^nM,N)=0$.
\end{lem}

\begin{proof} We replace $M$ with $\syz^{n-1}M$, and have $\Ext_R^{1}(\syz^{n-1} M,\syz N)=\Ext_R^{2}(\syz^{n-1}M,\syz N)=0$. Thus the result follows from Lemma \ref{prop}.
\end{proof}

\begin{dfn} \label{2rigid} Let $0\neq N\in \md R$. We say $N$ is 2-Tor-rigid if, whenever $0\neq M \in \md R$ and  $\Tor_{n}^R(M,N)=\Tor_{n+1}^R(M,N)=0$ for some $n\geq 0$, we have $\Tor_{i}^R(M,N)=0$ for all $i\geq n$.
\end{dfn}

There are various examples of 2-Tor-rigid modules in the literature. For example, over a hypersurface ring, each module in $\md R$ is 2-Tor-rigid \cite[1.9]{Mu}. Furthermore Dao \cite{Da1} pointed out that  finite length modules over codimension two complete intersection rings are 2-Tor-rigid. As Dao's manuscript \cite{Da1} is not published, we give a quick proof of this fact.

\begin{chunk} (Dao \cite[6.6]{Da1}) \label{Dao} Let $R$ be a complete intersection ring that is quotient of an unramified (e.g., equi-characteristic) regular local ring. Assume $R$ has codimension two. If $M\in \md R$ is a finite length module, then $M$ is 2-Tor-rigid.
\end{chunk}

\begin{proof} We may assume $R$ is complete so $R=S/(f,g)$, where $S$ is an unramified regular local ring and $\{f,g\}$ is a regular sequence on $S$ contained in the square of the maximal ideal of $S$. Let $M\in \md R$ be a finite length module and assume $\Tor^R_1(M,N)=\Tor^R_2(M,N)=0$ for some $N \in \md R$. Write $R=T/(g)$, where $T$ is the hypersurface ring $S/(f)$. It then follows from the standard long exact sequence \cite[11.64]{R} that $\Tor_2^S(M,N)=0$. Since $\dim_T(M)=\dim_R(M)=0$, and a finite length module over $T$ is Tor-rigid \cite[2.4]{HW}, we have $\Tor_i^S(M,N)=0$ for all $i\geq 2$. Using \cite[11.64]{R} once more, we conclude that $\Tor^R_i(M,N) \cong \Tor^R_{i+1}(M,N)$ for all $i\geq 1$. This gives the vanishing of $\Tor_i^R(M,N)$ for all $i\geq 1$.
\end{proof}

An important class of 2-Tor-rigid modules was determined by Burch \cite{B}. We record her result and use it for our proof of Corollary \ref{cor1}.

\begin{chunk} (Burch \cite[Theorem 5(ii), page 949]{B}) \label{Burc} Let $M\in \md R$ and let $I$ be a proper ideal of $R$. Assume $\fm (I:\fm)\neq \fm I$. If $\Tor_{n}^R(M,R/I)=\Tor_{n+1}^R(M,R/I)=0$ for some positive integer $n$, then $\pd(M)\leq n$.
\end{chunk}

The following consequence of Burch's theorem has been established in \cite{CelWag}. 

\begin{chunk} \label{Burc2} (Celikbas and Wagstaff \cite[2.3]{CelWag}) Let $M\in \md R$ and let $I$ be a proper ideal of $R$. Assume $I$ is  integrally closed and $\depth(R/I)=0$. If $\Tor_{n}^R(M,R/I)=\Tor_{n+1}^R(M,R/I)=0$ for some $n\geq 1$, then $\pd(M)\leq n$.
\end{chunk}

\begin{lem} \label{lemthm} Let $R$ be a local ring and let $M,N\in \md R$ be modules. Assume $N$ is nonfree and 2-Tor-rigid. If $\Ext_R^{n}(M,\syz N)=\Ext_R^{n+1}(M,\syz N)=0$ for some $n\geq 1$, then $\Tor_{j}^R(\tr\syz^nM,N)=0$ for all $j\geq 1$ and $\Ext^n_R(M,\syz^iN \oplus R)=0$ for all $i\geq 0$.
\end{lem}

\begin{proof} As $N$ is 2-Tor-rigid, Lemma \ref{lemma2rigid} implies that $\Tor_{j}^R(\tr\syz^nM,N)=0$ for all $j\geq 1$. Thus, given an integer $i\geq 0$, we have
$\Tor_2^R(\tr\syz^{n}M,\syz^iN)\cong\Tor_{2+i}^R(\tr\syz^{n}M,N)=0$ and $\Tor_1^R(\tr\syz^{n}M,\syz^iN)\cong\Tor_{1+i}^R(\tr\syz^{n}M,N)=0$. This yields, by the exact sequence (\ref{a1}.2), that:
\begin{equation} \tag{\ref{lemthm}.1}
\Ext_R^{n}(M,R)\otimes_R\syz^iN\cong\Ext_R^{n}(M,\syz^iN) \text{ for all } i\geq 0.
\end{equation}
Letting $i=1$, we obtain $\Ext_R^{n}(M,R)\otimes_R\syz N=0$, i.e., $\Ext_R^{n}(M,R)=0$. So (\ref{lemthm}.1) shows that $\Ext_R^{n}(M,\syz^iN)=0$ for all $i\geq 0$.
\end{proof}

We use properties of Gorenstein and complete intersection dimension for the rest of the paper. Therefore we recall the definitions of these homological dimensions.

\begin{chunk} \label{Gdim} (Auslander and Bridger \cite{AB})
A module $M\in \md R$ is said to be \emph{totally reflexive} if the
natural map $M \to M^{\ast\ast}$ is bijective and $\Ext_{R}^{i}(M,R)=0=\Ext_{R}^{i}(M^{\ast},R)$ for all $i\geq 1$. The infimum of $n$ for which there exists an exact sequence
$0 \to X_{n} \to \dots  \to X_{0} \to M\to 0,$
such that each $X_{i}$ is totally reflexive, is called the \emph{Gorenstein dimension} of $M$. If $M$ has Gorenstein dimension $n$, we write $\G-dim(M)=n$. Note, it follows by convention, that $\G-dim(0)=-\infty$.
\end{chunk}

\begin{chunk} \label{CIdim} (Avramov, Gasharov and Peeva \cite{AGP}) A diagram of local ring maps $R \to R' \twoheadleftarrow Q$ is called a {\em quasi-deformation} provided that $R \to R'$ is flat and the kernel of the surjection $R' \twoheadleftarrow Q$ is generated by a $Q$-regular sequence. The \emph{complete intersection dimension} of $M$ is:
\begin{equation*}
\CI(M) = \inf\{ \pd_Q(M \otimes_{R} R') - \pd_Q(R') \ | \
R \to R' \twoheadleftarrow Q \ {\text{is a quasi-deformation}}\}.
\end{equation*}
\end{chunk}

In the following $\H$ denotes a homological dimension of modules in $\md R$ that has the following properties; see \cite[3.1.2, 8.7 and 8.8]{Luchosurvey} for details.

\begin{chunk} \label{hdim} Let $M\in \md R$.
\begin{enumerate}[\rm(i)]
\item $\G-dim(M)\leq \H(M)\leq \CI(M) \leq \pd(M)$. \\If one of these dimensions is finite, then it equals the one on its left.
\item If $\H(M)<\infty$ and $M\neq 0$, then $\H(M)+\depth(M)=\depth(R)$.
\item If $\H(M)<\infty$, then $\H(M)=\sup \{i \in \ZZ: \Ext^i_R(M,R) \neq 0\}$.
\item If $\H(M)=0$, then $\Tr M = 0$ or $\H(\Tr M)=0$.
\item If $\H(M)<\infty$ and $0\neq N\in \md R$, then $\H(M)=\depth(N)-\depth(M\otimes_{R}N)$ provided that $\Tor_{i}^R(M,N)=0$ for all $i\geq 1$.
\end{enumerate}
\end{chunk}

The next result is important for our proof of Theorem \ref{l2}. 

\begin{chunk} (Celikbas, Gheibi, Sadeghi and Zargar \cite[5.8(1)]{CGSZ}) \label{CGSZ} 
Let $M\in \md R$ and let $n\geq 1$ be an integer. If $\G-dim(\Tr \syz ^n M)<\infty$, then $\Ext^i_R(\Tr \syz ^nM,R)=0$ for all $i=1, \dots, n$. 
\end{chunk}

Now we are ready to prove our main result:

\begin{thm}\label{l2}
Let $R$ be a local ring, $M,N \in \md R$ be modules and let $n\geq 1$ be an integer. Assume $N$ is not free and the following conditions hold:
\begin{enumerate}[\rm(i)]
\item $N$ is 2-Tor-rigid.
\item $\H(\tr\syz^n M)<\infty$.
\item $\Ext^n_R(M,\Omega N)=\Ext^{n+1}_R(M,\Omega N)=0$ for some $n\geq\depth(N)$.
\end{enumerate}
Then $\H(M)<n$.
\end{thm}

\begin{proof} It follows from (i), (iii) and Lemma \ref{lemthm} that $\Tor_i^R(\tr\syz^n M, N)=0$ for all $i\geq 1$. Hence, in view of \ref{hdim}(v), we deduce from (ii) that
\begin{equation}\tag{\ref{l2}.1}
\H(\tr\syz^n M)=\depth(N)-\depth(\tr\syz^n M\otimes_RN).
\end{equation}
Note, since $\H(\tr\syz^n M)<\infty$, we have $\G-dim(\tr\syz^n M)<\infty$; see \ref{hdim}(i).
Therefore, by (\ref{CGSZ}), we know that $\Ext^i_R(\tr\syz^n M,R)=0$ for all $i=1, \ldots, n$. Moreover, by (\ref{l2}.1), $\H(\tr\syz^n M)\leq n$. Consequently we deduce $\H(\tr\syz^n M)=0$; see \ref{hdim}(iii).

Since $\tr\tr \syz^n M$ is stably isomorphic to $\syz^n M$, we see from \ref{hdim}(iv) that $\syz^n M=0$ or $\H(\syz^n M)=0$. In either case, this implies $\H(M)\leq n$. Furthermore we have $\Ext^n_R(M,R)=0$ by Lemma \ref{lemthm}. Thus we conclude that $\H(M)<n$; see (\ref{hdim})(iii).
\end{proof}

\section{Corollaries of Theorem \ref{l2}}

This section is devoted to various corollaries of Theorem \ref{l2}. In particular we prove Theorem \ref{thmint} which is advertised in the introduction; see Corollary \ref{cor2}. To prove Corollaries \ref{corthm}, \ref{cornew2} and \ref{cornew3}, we use Theorem \ref{l2} for the cases where $\H$ is the projective dimension $\pd$ and the complete intersection dimension $\CI$; see \ref{CIdim}. We are allowed to do that since these homological dimensions satisfy the conditions listed in \ref{hdim}; see \cite[1.2]{A2} and \cite[2.5]{AY}.

\begin{cor} \label{corthm} Let $R$ be a local ring and let $M,N \in \md R$ be nonfree modules.
Assume $N$ is $2$-Tor-rigid and $\depth(N)\le1$. Assume further $\pd(M^{\ast})<\infty$. Then $\Ext_R^{1}(M,\syz N)$ and $\Ext_R^{2}(M,\syz N)$ do not vanish simultaneously.
\end{cor}

\begin{proof} Suppose $\Ext_R^{1}(M,\syz N)=\Ext_R^{2}(M,\syz N)=0$. Then it follows from Lemma \ref{lemthm} that $\Ext_R^1(M,R)=0$. Thus $\tr M$ is stably isomorphic to $\Omega \Tr \Omega M$; see (\ref{a1}.1). Hence $\pd(M^*)<\infty$ if and only if $\pd(\tr M)<\infty$ if and only if $\pd( \Omega \tr \Omega M)<\infty $ if and only if $\pd( \tr \Omega M)<\infty$. Consequently Theorem \ref{l2} shows that $M$ is free, which contradicts our assumption. Therefore $\Ext_R^{1}(M,\syz N)$ and $\Ext_R^{2}(M,\syz N)$ do not vanish simultaneously.
\end{proof}

\begin{cor} \label{cornew2} Let $R$ be a local ring, $M,N \in \md R$ be nonfree modules and let $n\geq 1$ be an integer. Assume $N$ is $2$-Tor-rigid and $\grade(M)\geq n$. If either $\Hom(M,N)\neq 0$ or $\depth(N)\le n$, then $\Ext_R^{n}(M,\syz N)$ and $\Ext_R^{n+1}(M,\syz N)$ do not vanish simultaneously.
\end{cor}

\begin{proof}
We assume $\Ext_R^{n}(M,\syz N)=\Ext_R^{n+1}(M,\syz N)=0$ and show that $\Hom(M,N)=0$ and $\depth(N)>n$. 

Note, since $\grade(M)\geq n$, we have $\Ext^i_R(M,R)=0$ for all $i=0, \ldots, n-1$.
Furthermore it follows from Lemma \ref{lemthm} that $\Tor_{j}^R(\tr\syz^nM,N)=0$ for all $j\geq 1$ and $\Ext^n_R(M, R)=0$. Consequently $\Ext^i_R(M,R)=0$ for all $i=0, \ldots, n$.
This implies $\tr \syz^{i-1}M$ is stably isomorphic to $\syz \tr \syz^{i}M$ for all $i=1, \ldots, n$; see (\ref{a1}.1). Using this fact, we deduce the vanishing of $\Tor_{j}^R(\tr M,N)$, for all $j\geq 1$, from the vanishing of $\Tor_{j}^R(\tr\syz^nM,N)$. Now (\ref{a1}.2) gives $\Hom(M,N)=0$.

Now let $\cdots\rightarrow F_1\rightarrow F_0\rightarrow0$ be a free resolution of $M$. Since  $\Ext^i_R(M,R)=0$ for all $i=0, \ldots, n$, we obtain an exact sequence of the form:
$$0\rightarrow (F_0)^{\ast} \rightarrow \cdots \rightarrow (F_{n+1})^{\ast} \rightarrow\tr\syz^n M\rightarrow0.$$
Therefore $\pd(\tr\Omega^nM)<\infty$. Hence, if $\depth(N)\le n$, then it follows from  Theorem \ref{l2} that $\pd(M)<n$ so that the fact $n\leq \grade(M)\leq \pd(M)$ gives a contradiction. Consequently $\depth(N)>n$.
\end{proof}

In passing we obtain a result on the vanishing of Ext over complete intersections:

\begin{cor} \label{cornew3} Let $R$ be a complete intersection ring of codimension two that is quotient of an unramified regular local ring. If $M \in \md R$ is a module that is not maximal Cohen-Macaulay and $I$ is an $\fm$-primary ideal of $R$, then $\Ext_R^{1}(M,I)$ and $\Ext_R^{2}(M,I)$ do not vanish simultaneously.
\end{cor}

\begin{proof} Assume $M$ is not maximal Cohen-Macaulay and $I$ is $\fm$-primary. Set $N=R/I$. Then $N$ is a nonfree finite length module that is $2$-Tor-rigid; see (\ref{Dao}). If $\Ext_R^{1}(M,I)=\Ext_R^{2}(M,I)=0$, then it follows from Theorem \ref{l2} that $\CI(M)<1$, i.e., $M$ is maximal Cohen-Macaulay; see (\ref{hdim})(ii). Hence $\Ext_R^{1}(M,I)$ and $\Ext_R^{2}(M,I)$ do not vanish simultaneously.
\end{proof}

Before we giving the definition of a weakly $\fm$-full ideal, we record an example that shows the finiteness hypothesis of $\pd(\tr\syz^n M)$ is necessary in Theorem \ref{l2} to conclude $\pd(M)<n$.

\begin{eg} \label{exvanish} Let $R=\CC[[x,y,z]]/(xz-y^2,xy-z^2)$, $M=R/(x,z)$ and $N=R/(x,y)$. Then $R$ is a codimension two complete intersection ring of dimension one, and $M$ and $N$ are finite length $R$-modules such that $\pd(M)=\infty=\pd(N)$. Furthermore $\Ext^i_R(M,N)=0=\Tor_i^R(M,N)$ for all $i\geq 2$; see \cite[4.2]{Jor}. Hence, setting $J=\Omega N$, we obtain $\Ext_R^i(M,J)=0$ for all $i\geq 3$. 
Let $n\geq 1$ and set $X=\Omega^nM$. Note $X^{\ast} \neq 0$. Assume $\pd(\Tr X)<\infty$. Then  $\pd(X^{\ast})<\infty$ so that $X^{\ast}$ is free. Since $X \cong X^{\ast\ast}$, we see $X$ is free, i.e., $\pd(M)<\infty$. Therefore $\pd(\Tr\Omega^nM)=\infty$ for all $n\geq 1$.
\end{eg}

\begin{rmk} In Example \ref{exvanish}, we can use either Corollary \ref{corthm} or Corollary \ref{cornew3} to see $\Ext_R^{1}(M,J)$ and $\Ext_R^{2}(M,J)$ do not vanish simultaneously. In fact we have $\Ext_R^{1}(M,J)\neq 0=\Ext_R^{2}(M,J)$.
\end{rmk}

Recall that an ideal $I$ of a local ring $R$ is called $\fm$-full if $\fm I:x=I$ for some $x\in \fm$; see \cite[2.1]{Goto}. Properties of such ideals were extensively studied  in the literature. For example Goto established that $\fm$-full ideals are closely related to integrally closed ideals:

\begin{chunk} (Goto \cite[2.4]{Goto}) \label{GR}
Assume the residue field of $R$ is infinite. If $I$ is an integrally closed ideal of $R$, then $I=\sqrt{(0)}$ or $I$ is $\fm$-full.
\end{chunk}

A weakly version of $\fm$-full property can be defined as follows:

\begin{dfn} \label{wmfulldfn} Let $I$ be a proper ideal of a local ring $R$. We call $I$ a \emph{weakly $\fm$-full ideal} provided that $\fm I: \fm =I$, i.e., if $\fm x \subseteq \fm I$ for some $x\in R$, then $x\in I$.
\end{dfn}

\begin{eg} \label{egprime} Let $I \neq \fm$ be a prime ideal of $R$. Let $x\in \fm I: \fm$, $y\in \fm$ and $y \notin I$. Then $xy\in x \fm \subseteq \fm I \subseteq I$ and hence $x\in I$. Therefore $I$ is weakly $\fm$-full.
\end{eg}

We should remark that, due to its definition, it can be easily checked whether or not a given ideal is weakly $\fm$-full by using a computer algebra software such as Macaulay2 \cite{M2}.

\begin{rmk} \label{weaklyrmk} Let $I$ be a proper ideal of a local ring $R$. Then $\depth(R/I)=0$ if and only if $I:\fm \neq  I$. Hence it follows that, if $\depth(R/I)\geq 1$, then $I$ is weakly $\fm$-full. 
\end{rmk}

Although each $\fm$-full ideal is weakly $\fm$-full, many weakly $\fm$-full ideals that are not $\fm$-full exist. We record such an example next.

\begin{eg} (\cite{CGTT}) \label{CGTTex} Let $R=\CC[\![t^4,t^5,t^6]\!]$ and let $I=(t^4,t^{11})$. Then $I$ is a weakly $\fm$-full ideal which is not $\fm$-full such that $\depth(R/I)=0$.
\end{eg}

Let $I$ be a nonzero ideal of a local ring $R$. Assume $\fm (I:\fm)=\fm I$ and $I$ is weakly $\fm$-full. Let $x\in I:\fm$. Then $\fm x \in \fm(I:\fm)=\fm I$, i.e., $x\in \fm I:\fm$, which equals to $I$ since $I$ is weakly $\fm$-full. So we conclude:

\begin{chunk} \label{wmfull} Let $I$ be a nonzero ideal of a local ring $R$. Assume $I$ is weakly $\fm$-full and $\depth(R/I)=0$. Then $\fm (I:\fm)\neq \fm I$.
\end{chunk}

Next we note that weakly $\fm$-full ideals $I$ with $\depth(R/I)=0$ are $2$-Tor-rigid; see (\ref{2rigid}).

\begin{chunk} \label{wmfull2} Let $I$ be a nonzero ideal of a local ring $R$.  Assume $I$ is weakly $\fm$-full and $\depth(R/I)=0$. If $\Tor_{n}^R(M,R/I)=\Tor_{n+1}^R(M,R/I)=0$ for some $M\in \md R$ and for some integer $n\geq 1$, then $\pd(M)\leq n$; see (\ref{Burc}) and (\ref{wmfull}).
\end{chunk}

\begin{cor}\label{cor1}
Let $M \in \md R$ be a nonzero module and let $I$ be a nonzero ideal of $R$ such that $\depth(R/I)=0$. Assume $\Ext^n_R(M,I)=\Ext^{n+1}_R(M,I)=0$ for some $n\geq 1$. If $I$ is integrally closed or weakly $\fm$-full, then $\pd(M)<n$.
\end{cor}

\begin{proof} Setting $N=R/I$, we see from Lemma \ref{lemthm} that $\Tor_{i}^R(\tr\syz^nM,N)=0$ for all $i\geq 1$ and $\pd(\tr\syz^nM)\leq 1$; see \ref{Burc2} and (\ref{wmfull2}). So the result follows from Theorem \ref{l2}.
\end{proof}

We now establish Theorem \ref{thmint}, advertised in the introduction.

\begin{cor}\label{cor2}
Let $I$ be a proper ideal of a local ring $R$ such that $\depth(R/I)=0$, or equivalently $I:\fm \neq I$. Assume $I$ is integrally closed or weakly $\fm$-full. Then,
\begin{enumerate}[\rm(i)]
\item $\Ext^1_R(I,I)$ and $\Ext^{2}_R(I,I)$ do not vanish simultaneously unless $I \cong R$. 
\item $\Ext^n_R(I,I)$ and $\Ext^{n+1}_R(I,I)$ do not vanish simultaneously for any $n\geq 1$ unless $R$ is regular.
\end{enumerate}
\end{cor}

\begin{proof} Part (i) is clear from Corollary \ref{cor1}. Next we assume $\Ext^n_R(I,I)=\Ext^{n+1}_R(I,I)=0$ for some $n\geq 1$. Then it follows from Corollary \ref{cor1} that $\pd(R/I)<\infty$. Thus $\Tor_i^R(k,R/I)=0$ for all $i\gg 0$. This implies $\pd(k)<\infty$, i.e., $R$ is regular; see (\ref{Burc2}) and (\ref{wmfull2}). 
\end{proof}

\begin{rmk} The conclusion of Corollary \ref{cor2}(i) resembles a result of Jorgensen: if $R$ is a a local complete intersection ring and $M\in \md R$, then $\Ext^1_R(M,M)$ and $\Ext^{2}_R(M,M)$ do not vanish simultaneously unless $M$ is free; see \cite[2.5]{JorExt}. In that sense, a weakly $\fm$-full ideal $I$ with $I:\fm \neq I$ behaves like a nonfree module over a complete intersection ring.
\end{rmk}

The vanishing result obtained in Corollary \ref{cor2} is stronger than what is proposed in the Auslander and Reiten conjecture: if $I$ is an ideal as in Corollary \ref{cor2}, then we do not need the vanishing of $\Ext_R^i(I,R)$ for any $i\geq 1$, besides the vanishing of $\Ext_R^1(I,I)$ and $\Ext_R^2(I,I)$, to conclude $I$ is free; c.f., Conjecture (\ref{ARconj}). In fact, for such an ideal $I$, over a Cohen-Macaulay local ring that is not Gorenstein, there cannot be $\dim(R)+2$ consecuive vanishing of $\Ext_R^i(I,R)$; this is recorded in the next corollary.

\begin{cor} \label{ABS} Let $R$ be a $d$-dimensional Cohen-Macaulay local ring that is not Gorenstein, and let $I$ be a proper ideal of $R$ such that $\depth(R/I)=0$, or equivalently $I:\fm \neq I$. Assume $I$ is integrally closed or weakly $\fm$-full. Then, for all $n\geq 1$, there is an integer $i$ with $n\leq i \leq n+d+1$ such that $\Ext^i_R(I,R)\neq 0$. 
\end{cor}

\begin{proof} Let $N\in \md R$ be a module of finite injective dimension. If, for some $n\geq 1$, we have $\Ext^i_R(I,R)=0$ for all $i=n, \ldots, n+d+1$, then \cite[Corollary B.4]{ABS} gives the vanishing of $\Tor_n^R(I,N)$ and $\Tor_{n+1}^R(I,N)$. Hence we conclude from \ref{Burc2} and \ref{wmfull2}  that $\pd(N)<\infty$ so that $R$ is Gorenstein. 
\end{proof}

Goto and Hayasaka \cite[2.2]{GH2} proved that if $R$ is a local ring and $I$ and $J$ are two ideals of $R$ such that $I\nsubseteq J$, $I:\fm \nsubseteq J$ and $J$ is $\fm$-full, then $R$ is regular provided that $\id(I)<\infty$. A special, albeit an important case of this result, namely the case where $I=J$, is:

\begin{chunk} (Goto and Hayasaka \cite[2.2]{GH2}) \label{GH} Let $R$ be a local ring and let $I$ be a proper ideal of $R$. Assume $\depth(R/I)=0$ and $I$ is $\fm$-full. If $\id(I)<\infty$, then $R$ is regular.
\end{chunk}

We use Corollary \ref{cor1} and improve \ref{GH} by replacing an $\fm$-full ideal with a weakly $\fm$-full one:

\begin{cor} \label{cor3} Let $R$ be a local ring and let $I$ be a proper ideal of $R$. Assume $\depth(R/I)=0$ and $I$ is weakly $\fm$-full. If $\id(I)<\infty$, then $R$ is regular.
\end{cor}

\begin{proof} As $\id_R(I)<\infty$ we have $\Ext^i_R(k,I)=0$ for all $i\gg0$. This implies, by Corollary \ref{cor1}, that $\pd(k)<\infty$, i.e., $R$ is regular.
\end{proof}

We can obtain a result similar to Corollary \ref{cor2} that characterizes Gorenstein rings in terms of the \emph{Gorenstein injective dimension} $\Gid(I)$ of weakly $\fm$-full ideals $I$; see \cite{Henrik} for the definition of Gorenstein injective dimension.

\begin{cor} Let $R$ be a local ring and let $I$ be a proper ideal. Assume $\depth(R/I)=0$ and $I$ is weakly $\fm$-full. If $\Gid(I)<\infty$ and $\dim(R/I)<\dim R$, then $R$ is Gorenstein.
\end{cor}

\begin{proof}
Note that $\dim(I)=\dim R$. Therefore it follows from \cite[1.3]{Yas} that $R$ is Cohen-Macaulay and so there exists a module $0\neq M\in \md R$ such that $\id(M)<\infty$. Hence $\Ext^i_R(M,I)=0$ for all $i\gg 0$ \cite[2.22]{Henrik}. Now, by Corollary \ref{cor1}, we have  $\pd(M)<\infty$ and so $R$ is Gorenstein.
\end{proof}

Now we present several examples to prove the necessity of the assumptions of our results.

\begin{eg} \label{eg1} Let $R=\CC[[x,y,z]]/(xy-z^2)$ and let $I=(x)$ be the ideal of $R$. Then $R$ is a two-dimensional integrally closed Gorenstein domain that is not regular. This implies, since $I$ is principal, that $I$ is integrally closed. Therefore $I$ is $\fm$-full and hence weakly $\fm$-full; see (\ref{GR}). Moreover $\id(I)=2$ and  $\depth(R/I)=1$.
\end{eg}

\begin{eg}  \label{eg2}
Let $S=k[[t,u]]$ be a formal power series ring over a field $k$, and let $R=k[[t^3,t^4,t^5,u]]$ be a subring of $S$. Then the ideal $I=(t^3,t^4)$ of $R$ is a canonical module of $R$, which implies $\id(I)=2$ and $\Ext^i_R(I,I)=0$ for all $i\geq 1$. Since $\depth(R/I)=1$, the ideal $I$ is weakly $\m$-full; see (\ref{weaklyrmk}). In fact we can show $I$ is $\fm$-full but not integrally closed. Since $(t^5)^2-(t^3)^2t^4=0$, it is clear that $I=(t^3,t^4)$ is not an integrally closed ideal of $R=k[[t^3,t^4,t^5,u]]$. Considering the residue ring $R/\m I$, we can easily check that the equality $I=(\m I:u)$ holds. Hence $I$ is $\m$-full.
\end{eg}

\begin{eg}\label{it}
Let $S=k[[t]]$ be a formal power series ring over a field $k$, and let $R=k[[t^5,t^6,t^8,t^9]]$ be a subring of $S$.
Then $R$ is a non-Gorenstein Cohen--Macaulay local domain of dimension $1$.
Denote by $\m$ the maximal ideal of $R$.
Then $\m$ is integrally closed, and is (weakly) $\m$-full.
The ideal $I=(t^5,t^8,t^9)$ of $R$ is a canonical module of $R$, and we have $\id(I)=1$ and $\depth(R/I)=0$.
We have $\id(\m)=\infty$; this follows from the fact that $\m$ is an indecomposable maximal Cohen--Macaulay $R$-module, which is not isomorphic to the canonical module $I$.
The equality $(t^6)^3-(t^5)^2(t^8)=0$ shows that $I$ is not integrally closed.
Calculating the socle of $R/\m I$, we can verify that $I$ is not weakly $\m$-full.
\end{eg}

The ideal $I$ in Example \ref{it} guarantees that the assumption in Corollaries \ref{cor1}, \ref{cor2} (and Theorem \ref{thmint}) that $I$ is either integrally closed or weakly $\m$-full is necessary.
It also assures that the weak $\m$-fullness assumption in Corollary \ref{cor3} cannot be removed.
The ideals $I$ in Examples \ref{eg1}, \ref{eg2} show the necessity of the depth assumption in Corollaries \ref{cor1}, \ref{cor2}, \ref{cor3} (and the assumption $I:\m\ne I$ in Theorem \ref{thmint}).
Also, the maximal ideal $\m$ in Example \ref{it} says that the finiteness assumption of injective dimension in Corollary \ref{cor3} cannot be extracted.
Thus, each of the three assumptions in Corollary 3.18 is necessary.

\section*{Acknowledgments}
We are grateful to C\u{a}t\u{a}lin Ciuperc\u{a} for useful discussions on $\fm$-full ideals and his suggestions for the definition of weakly $\fm$-full ideals. We are also grateful to Shiro Goto and Naoki Taniguchi for their valuable comments related to the manuscript.

\end{document}